\documentclass{amsart}

\newcommand{\cl}{\mathrm{cl}}

\newcommand{\tp}{\mathrm{tp}}

\newcommand{\R}{\mathbb{R}}

\newcommand{\EGen}[2]{\mathrm{Gen}_{#1}(#2^{ext})}

\newcommand{\Sext}[2]{S_{ext,#1}(#2)}
\newcommand{\GC}{\mathfrak{C}}

\usepackage{amsfonts}
\usepackage{amsmath}
\usepackage{amssymb}
\usepackage{amsthm}
\usepackage{enumerate}
\usepackage[utf8]{inputenc}
\usepackage{mathabx,epsfig}
\usepackage[bookmarks=false]{hyperref}

\newtheorem{theorem}{Theorem}
\numberwithin{theorem}{section}
\newtheorem{lemma}[theorem]{Lemma}

\newtheorem{fact}[theorem]{Fact}
\newtheorem{proposition}[theorem]{Proposition}
\newtheorem{remark}[theorem]{Remark}
\newtheorem{conjecture}[theorem]{Conjecture}

\newtheorem{question}[theorem]{Question}
\newtheorem{corollary}[theorem]{Corollary}
\newtheorem{definition}[theorem]{Definition}

\theoremstyle{definition}

\title{The Ellis group conjecture and variants of definable amenability}
\author{Grzegorz Jagiella
	}
\thanks{This research was supported by the National Science Centre, Poland grant no. 2014/13/N/ST1/02521}
\thanks{Partially supported by ValCoMo (ANR-13-BS01-0006)}
\thanks{This research was supported by the Israel-US Binational Science Foundation}
\begin{document}
\maketitle
\begin{abstract}
We consider definable topological dynamics for $NIP$ groups admitting certain decompositions in terms of specific classes of definably amenable groups. For such a group, we find a description of the Ellis group of its universal definable flow. This description shows that the Ellis group is of bounded size. Under additional assumptions, it is shown to be independent of the model, proving a conjecture proposed by Newelski. Finally we apply the results to new classes of groups definable in o-minimal structures, generalizing all of the previous results for this setting.
\end{abstract}

In this paper we work within the framework of definable topological dynamics. The reader is referred to \cite{New1} and \cite{New2} as the seminal papers. Familiarity with the subject will be assumed throughout the paper, but we will recall the necessary notions in the preliminaries. The main motivation for this research comes from an open problem in definable topological dynamics regarding the model-theoretic aspects of the Ellis group of a definable flow. Given a model $M$ and an $M$-definable group $G$, we consider the category of definable $G$-flows over $M$. These flows are $G(M)$-flows in the sense of classic topological dynamics. The category contains a universal object $\Sext{G}{M}$, the space of external types in $G$ over $M$. In the studies of this flow, a conection has been found relating its Ellis group to the model-theoretic connected components of $G$.

In \cite{New1}, Newelski conjectured that (at least under some ``tame'' assumptions), the Ellis group of the universal flow of a group $G$ is isomorphic to the quotient $G/G^{00}$, and that in particular it does not depend on the model. This was proved to be the case in stable theories \cite{New3} and for definably compact groups definable in o-minimal structures \cite{New2}. This result was later extended to definably amenable groups definable in o-minimal structures \cite{PY} and finally to all definably amenable groups in $NIP$ theories \cite{CP}. Partial results also hold in the general case. It was proven that in general the Ellis group factors through $G/G^{00}$ whenever $G^{00}$ exists \cite{New2}. Indeed, this result holds also with $G^{000}$ in place of $G^{00}$. On the other hand, there are examples of groups definable in a relatively tame o-minimal setup where the conjecture fails.

A related, weaker conjecture states that the Ellis group does not depend on the choice of the model. More precisely, given an $M$-definable group $G$ and $N \succ M$, one can consider the $G(M)$-flow $\Sext{G}{M}$ as well as the $G(N)$-flow $\Sext{G}{N}$. The conjecture asserts that the Ellis groups calculated for each of the flows are isomorphic, and that the witnessing isomorphism can be constructed in some definable way. This is the case whenever the original conjecture holds as the quotient $G/G^{00}$ depends only on $G$.

A large part of the study of definable topological dynamics has been conducted in the o-minimal setup. A study of $SL(2,\R)$ in \cite{GPP} provided a counterexample to the original Ellis group conjecture. In \cite{Jag}, a wide range of counterexamples have been produced by calculating the Ellis groups for the flows of groups definable in o-minimal expansions of the reals admitting definable compact-torsion-free decomposition. These results have been generalized in \cite{YL} to allow the calculation of the Ellis groups over larger models, establishing the isomorphism with groups calculated over the reals. A tangent case of definably amenable groups definable in o-minimal expansions of arbitrary real closed fields have been solved in \cite{PY}. The methods used to determine the Ellis groups have been progressively less specific to the o-minimal setting. 

In this paper, we replace the notions specific to the o-minimal setting with more robust, model-theoretic ones, using the research done by Chernikov and Simon on definably amenable $NIP$ groups. We provide a way of computing the Ellis groups for $NIP$ groups that either contain a definable, definably extremely amenable normal subgroup, or admit a definable decomposition into an $fsg$ and a definably extremely amenable subgroup. We then apply our results to the o-minimal case, generalizing all the previously obtained results. We finally discuss other generalizations or applications. Our main results for the o-minimal case is the following:

\begin{theorem}
	\label{TheoremOMin}
	Let $G$ be a definably connected group definable in an o-minimal expansion of a real closed field. Then over any model $M$ the Ellis group of the flow $\Sext{G}{M}$ is abstractly isomorphic to a subgroup of a compact Lie group.
\end{theorem}

In some cases we also note definability of this isomorphism, establishing the weaker Ellis group conjecture there.

The paper is divided into six sections. In the first section we recall the usual notions of definable topological dynamics and some general model theory, and cite some of the more important results that we use. In the second section we discuss the notion of definable amenability and its specific cases. In Section 3, we prove results regarding the Ellis group of the universal definable flow of a definable group in a $NIP$ theory that has a normal, definable, definably extremely amenable subgroup. In Section 4, again assuming $NIP$, we consider the dynamics of a group admitting a definable decomposition into an $fsg$ and a definably extremely amenable subgroups. In Section 5, we apply the results to the case of groups definable in o-minimal setting. In Section 6, we discuss generalizations and further applications.

\section{Preliminaries}
Throughout the paper, we use standard model-theoretic notations. Working over a fixed model of an ambient theory, we will write $\GC$ to denote a sufficiently saturated and homogeneous elementary extension. We assume the reader's familiarity with the basics of model theory, including heirs, coheirs and the notion of definability of types. By ``definable'' we always mean ``definable with parameters''. In the following subsections we discuss the notion of $NIP$ groups, recall the basic notions of both classic and definable topological dynamics, and the notion of definably amenable groups and their specific subclasses.

\subsection{$NIP$ theories and external definability}
In our investigations, we will be dealing with the notion of externally definable subsets and external types. Let $M$ be a structure. Recall that an externally definable subset of $M^n$ is a trace in $M^n$ of a formula with parameters from an elementary extension $N \succ M$. The set of all externally definable subsets of $M^n$ forms a boolean algebra $Def_{ext,n}(M)$. The set $S_{ext,n}(M) = S(Def_{ext,n}(M))$ of ultrafilters on $Def_{ext,n}(M)$ carries the Stone topology, and we call its elements external types. If $X$ is an $M$-definable subset, we write $\Sext{X}{M}$ for the closed subset of elements of $\Sext{n}{M}$ that contain $X$.

Working with external types in arbitrary theories is rather difficult. An assumption on the ambient first-order theory makes their structure more easily understood. Recall the following definition:
\begin{definition}
	A complete theory $T$ has $NIP$ if it contains no formula $\phi(\bar{x},\bar{y})$ with the following property: for every model $M \models T$, for each $n \in \omega$ there are tuples $b_0, \ldots, b_{n-1} \in M$ such that for each subset $X$ of $\{0, \ldots, n-1\}$ there is a tuple $a \in M$ such that $M \models \phi(a,b_i) \iff i \in X$.
\end{definition}
Now assume that $T$ is a $NIP$ theory in the language $\mathcal{L}$ and $M$ its model. For each externally definable set $X \in Def_{ext}(M)$ (for all $n$), let $D_X$ be a new predicate interpreted in $M$ such that $D_X(M) = X$. Let $\mathcal{L}_{ext,M} = \mathcal{L} \cup \{D_X : X \in Def_{ext}(M)\}$ be a new language and let $M^{ext}$ be the structure with the universe $M$ considered in the language $\mathcal{L}_{ext,M}$. Shelah proved in \cite{She} that:
\begin{proposition}
	Assume that $T$ has $NIP$ and $M \models T$. Then the first order $\mathcal{L}_{ext,M}$-structure $M^{ext}$ has elimination of quantifiers and all types over $M^{ext}$ are definable.
\end{proposition}
As a consequence of definability of types over $M^{ext}$, every type $p \in S(M^{ext})$ has a unique heir and a unique coheir over any set of parameters $A \supset M$. We will denote the unique heir of $p$ over $A$ by $p|A$. For brevity, if $\bar{a}$ is a finite tuple, we will write $p|\bar{a}$ instead of $p|M\bar{a}$. We will also employ the following notation, used for example in \cite{HP}, to easily express the heir and coheir relationships of elements:

\begin{definition}
	Let $(p_0, \ldots, p_{n-1})$ be a sequence of definable types. We write
	$$(a_0, \ldots, a_{n-1}) \models p_0 \otimes p_1 \otimes \ldots \otimes p_{n-1}$$
	to denote that for each $i < n$, $a_i \models p_i|a_{<i}$.
\end{definition}

By the elimination of quantifiers, the standard space of types $S_n(M^{ext})$ naturally identifies with the space of quantifier-free types $S_{qf,n}(M^{ext})$. Thus they can both be identified with the space $\Sext{n}{M}$ of external types in the original language.

In the later subsections, and in the paper in general, we will often start with a model $M$ of a $NIP$ theory and consider objects definable in the original language. Then in order to consider external types, we will pass to $M^{ext}$ where they can be identified with the standard types in $S(M^{ext})$. Since the universe of $M$ and $M^{ext}$ is the same, we will make no distinction between $X(M)$ and $X(M^{ext})$ for an $\mathcal{L}$-definable $X$.

Finally, we will use the following notation. Assume that $p$ is a global type finitely satisfiable in some model $M$. Then we write
$$p^M = \{\phi(M) : \phi \in p\}.$$
This is a external type in $S_{ext}(M)$.

\subsection{Definable topological dynamics}
First we briefly recall some basic notions and results of classic topological dynamics. Let $G$ be a group. A $G$-flow is a (left) action of $G$ on a compact, Hausdorff topological space $X$ by homeomorphisms. A $G$-flow is called point-transitive if it contains a dense $G$-orbit. A subflow $Y \subset X$ is a nonempty, closed $G$-invariant subset of $X$. A subflow is called minimal if it contains no proper subflows. With any $g \in G$ we can associate the corresponding function $\pi_g : X \rightarrow X$. Consider the space $X^X$ with pointwise convergence topology. The (compact) set $\cl(\{\pi_g : g \in G\}) \subset X^X$ is a point-transitive $G$-flow. Equipped with the function composition operation $*$, it is a semigroup. It is called the Ellis semigroup of the flow $X$, denoted $E(X)$. This semigroup operation is continuous on the first coordinate. We have that $E(E(X))$ is naturally isomorphic to $E(X)$.

Let $I$ be a minimal subflow of $E(X)$. Then $I$ is a minimal ideal of the semigroup $(E(X),*)$. Likewise, any minimal ideal of $(E(X),*)$ is a minimal subflow. Let $J(I)$ denote the set of idempotent elements of $I$. Then by general theory of compact semigroups,
$$I = \bigcup_{u \in J(I)} u * I.$$
Each $u * I$ is a group that we will call an ideal subgroup. The isomorphism class of an ideal subgroup does not depend on $u$ and the choice of the minimal ideal $I$. We call this isomorphism class the Ellis group of the flow $X$.

Turning now to the definable setting, we let $M$ be an arbitrary first-order structure and let $G = (G,\cdot)$ be an $M$-definable group. The space $\Sext{G}{M}$ is naturally acted upon by $G(M)$ by left translations and carries the Stone topology that makes it a compact, Hausdorff topological space. The set of all principal ultrafilters in $\Sext{G}{M}$ forms a dense orbit, making this Stone space a point-transitive $G(M)$-flow in the sense of classic topological dynamics. It is the universal definable flow of $G$ over $M$ in the sense of \cite{New1}. By \cite{New1}, this flow is naturally isomorphic to its own Ellis semigroup. As such, $\Sext{G}{M}$ is equipped with a semigroup operation.

Now assume that we work with a $NIP$ theory. We identify $\Sext{G}{M}$ with $S_G(M^{ext})$. Due to definability of types, the semigroup operation on $S_G(M^{ext})$ has the following, explicit definition. For $p,q \in S_G(M^{ext})$,
$$p * q = \tp(a \cdot b/M^{ext}),$$
where $a \models p$ and $b \models q|a$. Equivalently, $b \models q$ and $a$ satisfies the unique finitely satisfiable extension of $p$ over $b$; or simply $(a,b) \models p \otimes q$. In general, whenever $(a_0,\ldots,a_n) \models p_0 \otimes \ldots \otimes p_n,$ we have $a_0 \cdot \ldots \cdot a_n \models p_0 * \ldots * p_n$.

\subsection{Ellis group conjectures}
Let $G$ be an $M$-definable group. Recall that the model-theoretic connected component $G^{00}$ is the smallest type-definable subgroup of $G$ of bounded index. It is a normal subgroup of $G$. It is well-known that $G^{00}$ exists assuming $NIP$. Note that this is not true in general. Furthermore, by \cite{CPS}
\begin{fact}
	Let $T$ has $NIP$. Then $G^{00}$ calculated for $\mathcal{L}_{ext,M}$ exists and equals to $G^{00}$ calculated in $\mathcal{L}$.
\end{fact}

Newelski conjectured \cite{New1} that under some relatively ``tame'' assumptions (generally understood to include $NIP$), the following holds:
\begin{conjecture}
	The Ellis group of $\Sext{G}{M}$ is isomorphic to $G/G^{00}$.
\end{conjecture}

Another conjecture, generally found in \cite{New2} states:
\begin{conjecture}
	Let $N \succ M$. The Ellis group of $\Sext{G}{M}$ and the Ellis group of $\Sext{G}{N}$ are isomorphic.
\end{conjecture}
The stipulation of the second conjecture is that the isomorphism can be found in some definable way. Note that in general there is no obvious relationship between the types in $\Sext{G}{N}$ and $\Sext{G}{M}$ since there is no natural way in which the externally definable subsets of $M$ should be interpreted in $N$. Newelski \cite{New2} proposed the following solution to this problem. Let $N'$ be an elementary extension of $M_{ext}$ in the language $\mathcal{L}_{ext,M}$. Then $N = N'|\mathcal{L}$ is an elementary extension of $M$ in the original language such that every externally definable subset of $M$ has a natural interpretation in $N$. Extension of this kind is called $*$-elementary extension, and we write $M \prec^* N$ to denote it. Note that in this case types in $\Sext{G}{N}$ extend types in $\Sext{G}{M}$ since the language $\mathcal{L}_{ext,N}$ extends $\mathcal{L}_{ext,M}$. One can then state the following:
\begin{conjecture}
	There is an ideal subgroup in $\Sext{G}{N}$ whose restriction to $M$ is an ideal subgroup in $\Sext{G}{M}$.
\end{conjecture}
Note that in general, the restriction map $r : \Sext{G}{N} \rightarrow \Sext{G}{M}$ is not a semigroup homomorphism, so it is not known whether the image of an ideal subgroup by $r$ is an ideal subgroup, or whether an ideal subgroup in $\Sext{G}{M}$ is an epimorphic image of an ideal subgroup in some extension.

\section{Definable amenability}
In this section we recall facts about definably amenable groups and their specific cases. We also consider them from the point of view of topological dynamics.

The following definition appears for example in \cite{CS}:
\begin{definition}
	Let $G$ be a definable group. We say that $G$ is definably amenable if there is a finitely additive probabilistic measure on the algebra of the definable subsets of $G$ invariant under the group action.
\end{definition}
The measure stipulated in the definition is called a Keisler measure. Definably amenable $NIP$ groups are one of the large classes for which topological dynamics have been described in detail. Chernikov and Simon showed that the Ellis group conjecture holds in this setup \cite{CS}:
\begin{theorem}
	\label{TheoremNIPAmenable}
	Let $M$ be a model of a $NIP$ theory and $G$ be a definably amenable $M$-definable group. Let $I \subset \Sext{G}{M}$ be a minimal subflow and $u \in I$ an idempotent. Then the quotient map $G \rightarrow G/G^{00}$ restricted to $u*I$ is an isomorphism.
\end{theorem}

In this paper, we will consider groups described in terms of subgroups being specific cases of definably amenable groups. The group themselves will usually not be definably amenable. The motivation for the particular decompositions comes from the study of groups definable in o-minimal setting. The two particular cases we are interested in are $fsg$ groups and definably extremely amenable groups.

\subsection{Finitely satisfied generics}
The following definition can be found in \cite{HPP}:

\begin{definition}
	$G$ admits finitely satisfied generics (in short: ``$G$ has $fsg$'') if there is a global type $p(x)$ in $G$  and a small model $M$ such that every $G$-translate of $p$ is finitely satisfiable in $M$.
\end{definition}

In \cite{CS}, the $fsg$ groups are characterized as possessing a unique generic Keisler measure. In particular, $fsg$ groups are definably amenable. Examples of such groups are definably compact groups definable in o-minimal expansions of real closed fields, or over the field of $p$-adic numbers.

Recall that a subset $X \subset G$ is called generic if finitely many of its translates cover $G$. A type in $G$ is generic if it only contains formulas defining generic subsets. Likewise an external type in $G$ is generic if it only contains generic external subsets of $G$. The following properties of $fsg$ groups can be found in \cite{HPP}:
\begin{fact}
	\label{FactFSGGenerics}
	Let $G$ be $fsg$ and $M$ any small model. Then there is a generic global type $p \in S_G(\GC)$. Moreover, for any such a type:
	\begin{enumerate}
		\item Every left and right translate of $p$ is generic and finitely satisfiable in $M$.
		\item $G^{00}$ exists and is both the left and the right stabilizer of $p$.
	\end{enumerate}
\end{fact}
Clearly, any generic type is a finitely satisfiable extension (a coheir) of its own restriction to any submodel.

Let $p$ be a global type in $G$ finitely satisfiable in some small model $M$. Then the set $p^M = \{\phi(M) : \phi \in p\}$ is an element of $S_G(M^{ext})$. It is easy to see the following:
\begin{fact}
	\label{FactGenericLR}
	Let $G$ be $fsg$. Let $p \in S_G(\GC)$ be a global type. The following are equivalent:
	\begin{enumerate}[(i)]
		\item $p$ is generic.
		\item For some small model $M$, $p^M \in S_G(M^{ext})$ is generic.
		\item For any small model $M$, $p^M \in S_G(M^{ext})$ is generic.
	\end{enumerate}
\end{fact}
The above fact allows us to identify the set of generic types in $S_G(M^{ext})$ (for any $M$) with the set of global generic types. For a definable group $G$, denote by $\EGen{G}{M}$ the set of (left) generic types in $\Sext{G}{M}$.
From the dynamical point of view, Newelski proves in \cite{New1}:
\begin{proposition}
	\label{PropGenericUnique}
	Let $G$ be a group and assume that there is a generic $p \in \Sext{G}{M}$. Then the set $\EGen{G}{M}$ is the unique minimal subflow of $\Sext{G}{M}$.
\end{proposition}
Now let $G$ be an $M$-definable $fsg$ group in a $NIP$ theory. Combining Proposition \ref{PropGenericUnique} with Theorem \ref{TheoremNIPAmenable}, we easily see the following:
\begin{fact}
	The minimal flow $\EGen{G}{M}$ decomposes into ideal subgroups of the form $q * \EGen{G}{M}$ where $q$ is a generic with $q \vdash G^{00}$.
\end{fact}

We will need the following:
\begin{lemma}
	\label{LemmaSameIdealSG}
	Let $q \in \EGen{G}{M}$. Then for any $p \in \Sext{G}{M}$, $q*p$ is generic and both $q$ and $q*p$ belong to the same ideal subgroup.
\end{lemma}
\begin{proof}
	Write $q = r^M$ for some global generic $r$ and assume $q \in u * \EGen{G}{M}$ for some idempotent $u$. Let $b \models p$ and $a \models r^N$ for some $N \succ M$ containing $b$. Then $a\cdot b \models q*p$. Since $G^{00}(N) = Stab_R(r^N)$, we have $a \cdot b \equiv_M a \cdot b'$ for any $b'$ with $b'/G^{00}(N) = b/G^{00}(N)$ provided that $\tp(a/Mb')$ is finitely satisfiable. In particular $b'$ can be found satisfying a generic type. Thus $q * p = q * p' = u * q * p'$ for some generic $p'$ and so $q * p \in u * \EGen{G}{M}$.
\end{proof}

\begin{corollary}
	The flow $\EGen{G}{M}$ is a two-sided ideal of $\Sext{G}{M}$.
\end{corollary}

\subsection{Definable extreme amenability}
We now turn to discuss definably extremely amenable groups. We will use the following definition:
\begin{definition}
	A group $G$ is definably extremely amenable if it is definably amenable witnessed by a Keisler measure with the image $\{0,1\}$.
\end{definition}
It is easy to see that measures as in the definition correspond to complete global $G$-invariant types in $S_G(\GC)$. Thus:
\begin{fact}
	Let $G$ be definably extremely amenable. Then there is a $G$-invariant type $p \in S_G(\GC)$. Moreover, for any such type and any model $M$, the restriction $p \restriction M \in S_G(M)$ is $G(M)$-invariant.
\end{fact}
We first note that the property of having an invariant type is preserved when passing to $\mathcal{L}_{ext,M}$. By \cite{CPS}:
\begin{proposition}
	\label{PropExtInvariant}
	Let $G$ be definable and $M$ a model of a $NIP$ theory. Let $p \in S_G(M)$ be $G(M)$-invariant. Then there is a $G(M)$-invariant $p' \in \Sext{G}{M}$ extending $p$.
\end{proposition}

By for example \cite{Jag}:
\begin{lemma}
	\label{LemmaInvHeir}
	Assume that $p \in S_G(M)$ is a $G(M)$-invariant definable type. Then for any $N \succ M$, the unique heir $p|N$ is $G(N)$-invariant. In particular, the global heir of p is $G$-invariant.
\end{lemma}

\begin{corollary}
	Let $G$ be definable and $M$ a model of a $NIP$ theory. Assume that $p \in \Sext{G}{M}$ is $G(M)$-invariant. Then for any $M \prec^* N$ the type $p|N$ (the heir of $p$ over $N$ in $\mathcal{L}_{ext,M}$) is $G(N)$-invariant and extends to a $G(N)$-invariant $p' \in \Sext{G}{N}$.
\end{corollary}

Topological dynamics for definably extremely amenable groups is straightforward, as any $G(M)$-invariant type forms a one-point minimal flow that is its own unique ideal subgroup.

\section{Groups with normal definably extremely amenable subgroup}
\label{SectionDEA}
Assume that we work in a $NIP$ theory. Let $G$ be a definable group and $H \lhd G$ a definable normal subgroup. We will show that topological dynamics of $G$ is fully explained by dynamics of the quotient $G/H$.

Let $M$ be a model and assume that there is an $H(M)$-invariant type in $S_G(M)$. By Proposition \ref{PropExtInvariant} we may assume there is an $H(M)$-invariant external type $p \in \Sext{G}{M}$. For the remainder of the section, we fix $M$, $G$, $H$ and $p$.

The canonical quotient map $\pi_H : G \rightarrow G/H$ naturally extends to a map from the space of (external) types in $G$ over $M$ to the space of (external) types in ${G/H}$ over $M$. For a type $q$, write $q/H$ for its projection.

\begin{lemma} With the notation above,
	\begin{enumerate}[(i)]
		\item Let $q \in \Sext{G}{M}$. Then the type $q * p$ depends only on $q/H$.
		\item The set $\Sext{G}{M} * p$ is a subflow of $\Sext{G}{M}$ isomorphic to $\Sext{G/N}{M}$ via the projection map.
		\item There is a minimal subflow of $\Sext{G}{M} * p$ that projects isomorphically to a minimal subflow of $\Sext{G/N}{M}$.
	\end{enumerate}
\end{lemma}
\begin{proof}
	(i) Let $a \in G$ and $h \in H$. If $h' \models p|a,h$, then $ahh' \models \tp(ah/M) * p$. Since $hh' \models p|a,h$ by Lemma \ref{LemmaInvHeir}, we also have $ahh' \models \tp(a/M) * p$.
	
	(ii) Clearly $S_G(M) * p$ is an ideal of $\Sext{G}{M}$. Let $q_1, q_2 \in G$ and $a \models q_1$, $h \models p|a$, $b \models q_2|a,h$ and $h' \models p|a,h,b$. Then $ahbh' \models q_1 * p * q_2 * p$, but $ahbh' = abh''h'$ for some $h'' \in H$ since $N$ is normal in $G$. As $h''h' \models p|a,b,h$, we have $abh''h \models q_1 * q_2 * p$ as needed.
	
	(iii) follows from (ii).
\end{proof}

Whenever a minimal flow projects isomorphically onto a minimal flow, idempotents map to idempotents and their associated ideal subgroups also map isomorphically. As a corollary we obtain
\begin{proposition}
	\label{PropDEAIsom}
	Let $H \lhd G$ be definably extremely amenable. Then the Ellis groups of $(G(M),\Sext{G}{M})$ and $(G(M)/H(M),\Sext{G/H}{M})$ are isomorphic.
\end{proposition}

\section{$fsg$-definably extremely amenable decomposition}
In this section we consider the case of a group that properly decomposes into $fsg$ and definably extremely amenable subgroups. The motivation for this decomposition comes from the theory of definable Lie groups and represent a certain abstract definable version of Iwasawa decomposition. We will make this connection more clear in a later section.
All throughout this section, we assume to work in a $NIP$ theory. We begin with a suitable ad hoc definition.
\begin{definition}
	\label{DefinitionGood}
	Let $G$ be $M$-definable. We say that $G$ has a good decomposition if there are $M$-definable subgroups $K, H < G$ such that:
	\begin{enumerate}
		\item $G = KH$ and $F \cap H = \{1_G\}$.
		\item $K$ has $fsg$.
		\item $H$ is definably extremely amenable.
	\end{enumerate}
	In this case we will also say that $G=KH$ is a good decomposition.
\end{definition}

\begin{remark}
The condition (1) in the above definition is saying precisely that $G=KH$ is a Zappa-Sz\'ep decomposition.
\end{remark}
We will aim to describe the topological dynamics of $G$ admitting a good decomposition. Since we do not assume that the definably extremely amenable subgroup is normal, there is no straightforward reduction to dynamics of the quotient.

Assume that $(G,\cdot)$ is $M$-definable group in a $NIP$ theory and that $G=KH$ is a definable decomposition with $K \cap H = \{1_G\}$. The description of dynamical objects in $\Sext{G}{M}$ will involve a natural action of $H$ on $K$ induced by the decomposition. Let $g \in G$ be any element. Since the intersection of $K$ and $H$ is trivial, $g$ can be uniquely written as a product $kh$ with $k \in K, h \in H$, and likewise as a product $h'k'$ with $h' \in H, k' \in K$. The pairs $(k,h), (h',k')$ and $g$ are all interdefinable. Define an action of $H$ on $K$ as follows:
$$h \cdot_1 k = k' \iff hk = k'h' \text{ for some } h' \in H.$$
A direct computation shows this action is well-defined.

The action $\cdot_1$ lifts to an action $*_1$ of the semigroup of types $\Sext{H}{M}$ on the space of types $\Sext{K}{M}$ in the following way:
$$p *_1 q = \tp(h \cdot k/M)$$
for $h \models p, k \models q|h$. The fact this is a well-defined semigroup action follows from a more general variant of definable topological dynamics that we omitted in the introduction. We will instead show it directly.
\begin{lemma}
	Let $p',p \in \Sext{H}{M}$ and $q \in \Sext{K}{M}$. Then $p' *_1 (p *_1 q) = (p' * p) *_1 q.$
\end{lemma}
\begin{proof}
	Simply observe that if $(h',h,k) \models p' \otimes p \otimes q$, then $h' \cdot_1 (h \cdot_1 k) \models p' *_1 (p *_1 q)$ and also $(h' \cdot h) \cdot_1 k \models (p' * p) *_1 q$.
\end{proof}

We will also use the following:
\begin{lemma}
	\label{LemmaTypeP}
	Assume that $p \in \Sext{H}{M}$ is an $(M)$-invariant type. Then for any $q \in \Sext{K}{M}, p' \in \Sext{H}{M}$, we have $p' * q * p = (p' *_1 q) * p$.
\end{lemma}
\begin{proof}
	Let $(h',k,h) \models p' \otimes q \otimes p$. Then	$h'kh = (h' \cdot_1 k)h''h$ for some $h'' \in H$. Then by Lemma \ref{LemmaInvHeir} $h''h \models p|h'k$ and so
	$$(h' \cdot_1 k)h''h \models \tp(h' \cdot_1 k/M) * p = (p' *_1 q) * p,$$
	as required.
\end{proof}

We are now ready to approximate a minimal flow in the setting discussed.
\begin{proposition}
	\label{PropDecompFlows1}
	Let $G = KH$ be a definable (not necessarily good) decomposition with $K \cap H = \{1_G\}$ and let $p \in \Sext{G}{M}$ be $H(M)$-invariant. Then there is $I \subset \Sext{K}{M}$ such that $I * p$ is a minimal subflow of $\Sext{G}{M}$.
\end{proposition}
\begin{proof}
	It is sufficient to show that the set $\Sext{K}{M} * p$ is a subflow. Let $q * p \in \Sext{K}{M} * p$ and $k \in K(M), h \in H(M)$. Let also $(k',h') \models q \otimes p.$ Then
	$$khk'h' \models kh \cdot (q * p).$$
	We have $khk'h' = k(h \cdot_1 k') \cdot h''h$ for some $h'' \in H$. The triples $(k,h,k')$ and $(k, h \cdot_1 k', k'')$ are interdefinable over $M$, so $h \models p|k, h \cdot_1 k', k''$. By Lemma \ref{LemmaInvHeir}, also $h''h \models p|k, h \cdot_1 k', k''$ and so
	$$khk'h' = k(h \cdot_1 k') \cdot h''h \models \tp(k \cdot (h \cdot_1 k')/M) * p.$$
	Therefore we have
	$$G(M) \cdot q * p \subset \Sext{K}{M} * p.$$
	Since the semigroup operation is continuous on the first coordinate, this also implies
	$$\cl(G(M) \cdot q * p) \subset \Sext{K}{M} * p.$$
	This shows that $\Sext{K}{M} * p$ is closed. The same argument also shows it is $G(M)$-invariant.
\end{proof}

\begin{proposition}
	\label{PropDecompFlows2}
	Let $G = KH$ be a good decomposition and let $p \in \Sext{G}{M}$ be $H(M)$-invariant. Let $I * p$ be a minimal subflow as in Proposition \ref{PropDecompFlows1}. Then $I$ is unique and $\EGen{K}{M} \subset I$.
\end{proposition}
\begin{proof}
	First we show that $\EGen{K}{M} \subset I$. Let $q \in I$. Then the closure of the orbit $K(M)q$ contains $\EGen{K}{M}$ by Proposition \ref{PropGenericUnique}. Similarly, the closure of the orbit $K(M) \cdot (q * p) = (K(M) \cdot q) * p$ contains $\EGen{K}{M} * p$. For uniqueness, note that $I * p = \cl(G(M) \cdot \EGen{K}{M} * p)$.
\end{proof}

For the remainder of this section, fix an $M$-definable group $G$ and its good decomposition $G = KH$, as well as an $H(M)$-invariant type $p \in \Sext{H}{M}$. Let $I \subset \Sext{K}{M}$ be the unique set such that $I * p$ is a minimal subflow of $\Sext{G}{M}$. By general topological dynamics, $I * p$ decomposes into the disjoint union of its ideal subgroups. The set $\EGen{K}{M} * p \subset I * p$ is usually not closed. However:
\begin{proposition}
	\label{PropSameSG}
	$\EGen{K}{M} * p$ is a disjoint union of ideal subgroups of $I * p$. Moreover, if $q,q'$ are generic such that $q * p$ and $q' * p$ are in the same ideal subgroup of $I*p$, then $q$ and $q'$ and in the same ideal subgroup of $\EGen{K}{M}$ in the sense of the $K(M)$-flow $\Sext{K}{M}$.
\end{proposition}
\begin{proof}
	Let $q \in \EGen{K}{M}$. The type $q * p$ belongs to an ideal subgroup $\mathcal{H} \subset I * p$ and naturally $q * p * \mathcal{H} = \mathcal{H}$. Therefore $\mathcal{H} \subset q * p * I * p.$
	So let $r \in I$ and let $(k,h,k',h') \models q \otimes p \otimes r \otimes p$. Then $khk'h' = k(h \cdot_1 k')h''h$ for some $h'' \in H$. Arguing as in the proof of Proposition \ref{PropDecompFlows1}, we have
	$$k(h \cdot_1 k')h''h \models q * (p *_1 r) * p.$$
	By Lemma \ref{LemmaSameIdealSG}, $q' = q * (p *_1 r)$ is generic and $q,q'$ belong to the same ideal subgroup of $\EGen{K}{M}$.
\end{proof}

Using Lemma \ref{LemmaTypeP} we, give the following description on the semigroup operation in the minimal flow $I * p$:
\begin{fact}
	Let $q, q' \in I$. Then $(q * p) * (q' * p) = q * (p *_1 q') * p$.
\end{fact}

\begin{lemma}
	\label{LemmaSGForm}
	There is an ideal subgroup of $I * p$ of the form $J * p$ where $J = q_0 * (p *_1 \EGen{K}{M})$ for some $q_0 \in \EGen{K}{M}$. Moreover, $q_0$ can be chosen such that $q_0 * p$ is an idempotent.
\end{lemma}
\begin{proof}
	Take any $J \subset \EGen{K}{M}$ such that $J * p$ is an ideal subgroup. By Proposition \ref{PropSameSG} we may in fact assume that $J \subset u * \EGen{K}{M}$ for some idempotent $u \in \EGen{K}{M}$. Let $q_0 \in u * \EGen{K}{M}$ such that $q_0 * p \in J * p$ is an idempotent. Then $J * p = q_0 * p * J * p$. We have
	$$q_0 * p * J * p \subset q_0 * p * \EGen{K}{M} * p = q_0 * (p * _1 \EGen{K}{M}) * p.$$
	On the other hand,
	$$J * p = q_0 * p * I * p \supset q_0 * p * \EGen{K}{M} * p.$$
	Thus $J$ has the desired form.
\end{proof}

\begin{remark}
	Note that by Lemma \ref{LemmaSameIdealSG}, the elements of $J$ are generic and by Proposition \ref{PropSameSG} belong to the same ideal subgroup of $\EGen{K}{M}$.
\end{remark}

Now fix an ideal subgroup of $I * p$ the form $J * p$ as in Lemma \ref{LemmaSGForm} with an idempotent $q_0 * p$. Recall that the quotient map $\pi: K \rightarrow K^{00}$ is a group epimorphism. This epimorphism extends to the semigroup epimorphism $\pi : \Sext{K}{M} \rightarrow K/K^{00}$ and by Theorem \ref{TheoremNIPAmenable}, the restriction of $\pi$ to any ideal subgroup of $\EGen{K}{M}$ is a group isomorphism. For $q \in \EGen{K}{M}$, write $q/K^{00}$ to denote $\pi(q)$. Consequently:
\begin{lemma}
	\label{LemmaBijection}
	Let $i : J * p \rightarrow K/K^{00}$ be defined as follows:
	$$i(q*p) = q/K^{00}.$$
	Then $i$ is injective.
\end{lemma}
We now show that the group operation in $J * p$ can be described in terms of group operation in $K/K^{00}$. Define $\pi_K : G \rightarrow K$ as follows: $\pi_K(g) = k$ iff $g = kh$ for some $h \in H$. This function naturally lifts to $\pi_K : \Sext{G}{M} \rightarrow \Sext{K}{M}$. For any $q \in \Sext{K}{M}, q' \in \Sext{H}{M}$ we have $\pi_K(q * q') = q$.

With this notation, the map from Lemma \ref{LemmaBijection} is of the form $\pi \circ \pi_K.$

\begin{lemma}
	\label{LemmaShift}
	Let $k_0 = q_0/K^{00} \in K/K^{00}$. Then for any $q \in J$, $(p *_1 q)/K^{00} = k^{-1}_0 \cdot q/K^{00}.$
\end{lemma}
\begin{proof}
	We have $q * p = q_0 * p * q * p = q_0 * (p *_1 q) * p$. Thus $q/K^{00} = (q_0 * (p *_1 q))/K^{00} = q_0/K^{00} \cdot (p *_1 q)/K^{00}$.
\end{proof}
Fix $k_0 = q_0/K^{00}$. We show the following:
\begin{proposition}
	Let $K' = i[J * p] \subset K/K^{00}$ be equipped with the operation $a \circ b = a \cdot k^{-1} \cdot b$ with $\cdot$ denoting the group operation in $K/K^{00}$. Then $(K',\circ)$ is a group isomorphic with $J * p$ via $i$.
\end{proposition}
\begin{proof}
	It suffices to show that $i$ is a homomorphism. Let $q, q' \in J$. Then
	$$i(q * p * q' * p) = i(q * (p *_1 q') * p) = (q * (p *_1 q'))/K^{00} = q/K^{00} \cdot (p *_1 q')/K^{00}.$$
	Now by Lemma \ref{LemmaShift},
	$$q/K^{00} \cdot (p *_1 q')/K^{00} = q/K^{00} \cdot k^{-1} \cdot q'/K^{00} = i(q * p) \cdot k^{-1} \cdot i(q' * p).$$
\end{proof}
Finnaly, note that $(K',\circ)$ is isomorphic to the group $(K'k^{-1},\cdot)$ by sending $a$ to $ak^{-1}$. Thus we prove
\begin{theorem}
	\label{TheoremEllis}
	Let $G = KH$ be a good decomposition. Then the Ellis group of the universal definable flow $\Sext{G}{M}$ is isomorphic to a subgroup of $K/K^{00}$.
\end{theorem}

We now discuss how much our results depend on the chosen model $M$. We managed to describe the Ellis group as the object of the form $q_0 * (p *_1 \EGen{K}{M}) * p$ and showed how it maps to a subgroup of $K/K^{00}$. Particularly, this group depends on the $K^{00}$ cosets to which the operation $q \mapsto p *_1 q$ sends generic types in $K$.

Let $M \prec^* N$ and let $p|N$ be the unique heir. It is a type in $H$ in the language $\mathcal{L}_{ext,M}$ over $N$. By Proposition \ref{PropExtInvariant} there is some $p' \in \Sext{H}{N}$ extending $p|N$.

\begin{question}
	Let $q$ be a global generic in $K$. Is it true that $p *_1 q^M$ and $p' * q^N$ belong to the same coset of $K^{00}$?
\end{question}
Observe that if the above is true for all global generic types $q$, then the Ellis group does not depend on the model, that is $p *_1 \EGen{K}{M}$ intersects the same cosets of $K^{00}$ as $p' *_1 \EGen{K}{N}$ and $q^M * p$ is an idempotent if and only if $q^N * p'$ is. We do not know the answer to that question, or whether the negative answer implies that the Ellis group depends on the model. Here we list some examples of situations where the answer is positive:
\begin{itemize}
	\item[--] Whenever the action of $H$ on $K$ is trivial, $p *_1 q$ is simply $q$ for any type. This is the case precisely when $K$ normalizes $H$, that is when $G = K \ltimes H$.
	\item[--] If $G$ is definable over an o-minimal expansion of the field of reals. In this case, over any extension $R$, $(p' *_1 q^R)/K^{00}$ depends only on $q/K^{00}$. See \cite{YL} for more details.
	\item[--] If every global generic in $K$ is definable over a small model, then each $q^N$ is the unique external generic extending the unique heir of $q^M$. In this case, the $K^{00}$ coset of $p' * q^N$ is a definable property of the pair $(p,q)$.
	\item[--] If $K$ is stable and stably embedded in $G$, then every global generic in $K$ is definable and the preceding case holds.
\end{itemize}

We have the following specific result when the action of $p'$ on $\Sext{K}{N}$ respects the cosets of $K^{00}$, as is the case for groups definable in an o-minimal expansion of the reals.
\begin{proposition}
	\label{PropPreserveCoset}
	Assume that $(p' *_1 q^N)/K^{00}$ depends only on $q^N/K^{00} (= q/K^{00})$. Then the Ellis group of the universal definable flow of $G$ over $N$ contains $(N_G(H) \cap K)/K^{00}$ as a subgroup.
\end{proposition}
\begin{proof}
	If $k \in N_G(H) \cap K$, then $H$ fixes $k$ under the action $\cdot_1$. Thus for any external type $q$ such that $q/K^{00} = k/K^{00}$ we have $(p' *_1 q)/K^{00} = q/K^{00}$. In particular, this holds for a generic $q$. This shows that the set $p' *_1 \EGen{K}{N}$ meets $k/K^{00}$.
\end{proof}

\section{The o-minimal case}
In this section we apply the previous results to the case of groups definable in an o-minimal expansion of a real closed field. We refer the reader to \cite{PS} for classic results on such groups, including the definition of definable compactness. Here, we use the established results to prove Theorem \ref{TheoremOMin} and consider classes of groups where we can prove that the Ellis group does not depend on the model.

Let $R = (R, +, \cdot, <, 0, 1, \ldots)$ be an o-minimal expansion of a real closed field. Let $G$ be an $R$-definable group. We recall the following result from \cite{HPP}:
\begin{proposition}
	If $G$ is definably compact, then it is $fsg$.
\end{proposition}

Another result comes from \cite{CP}:
\begin{proposition}
	If $G$ is torsion-free, then it is definably extremely amenable.
\end{proposition}

We now turn to the structural theory developed by Conversano in \cite{Con1} and \cite{Con2} to describe an arbitrary group $G$.
\begin{definition}
	We say that $G$ has a definable compact-torsion-free decomposition if there are definable subgroups $K,H < G$ with $K \cap H = \{1_G\}$ such that $K$ is definably compact, and $H$ is torsion-free.
\end{definition}
Clearly a definable compact-torsion-free decomposition is a good decomposition in the sense of Definition \ref{DefinitionGood}.

Not every group has a definable compact-torsion-free decomposition. However, for any definably connected group $G$ it is possible to find a canonical quotient that admits a decomposition. After \cite{Con1}:
\begin{proposition}
	\label{PropDecomp}
	Let $G$ be a definable definably connected group. Then there exists a definable torsion-free $\mathcal{A}(G) \lhd G$ such that $G/\mathcal{A}(G)$ has a definable compact-torsion-free decomposition, and it is the maximal definable quotient of $G$ with a definable compact-torsion-free decomposition.
\end{proposition}
Since torsion-free groups are definably extremely amenable, this is the case we considered in Section \ref{SectionDEA}.

\begin{corollary}
	Theorem \ref{TheoremOMin} is true.
\end{corollary}
\begin{proof}
	Let $\mathcal{A}(G) \lhd G$ be as in Proposition \ref{PropDecomp} and let $G/\mathcal{A}(G) = KH$ be a definable compact-torsion-free decomposition. By Proposition \ref{PropDEAIsom} the Ellis group of $\Sext{G}{R}$ is isomorphic to the Ellis group of $\Sext{G/\mathcal{A}(G)}{R}$. Since $G/\mathcal{A}(G) = KH$ is a good decomposition, by Theorem \ref{TheoremEllis} the Ellis group is isomorphic to a subgroup of $K/K^{00}$. By the classic theorem of Pillay, $K/K^{00}$ with logic topology is is a compact Lie group.
\end{proof}
If $G/\mathcal{A}(G)$ is definably isomorphic to a group definable over the reals, this result can be improved. We have the following result by Yao \cite{YL}:
\begin{proposition}
	Let $G$ be definable over the reals with a compact-torsion-free decomposition $G = KH$. Then over any model the Ellis group of its universal definable flow is isomorphic to $N_G(H) \cap K(\mathbb{R})$.
\end{proposition}
Notice that over the reals, $N_G(H) \cap K(\mathbb{R})$ can be identified with the group $(N_G(H) \cap K)/K^{00}$. Compare this result with Proposition \ref{PropPreserveCoset}.

\section{Generalizations}
In this section we discuss some possible generalizations of our results. The motivation to consider the case of groups admitting a good decomposition comes from the compact-torsion-free decomposition theorem that allowed us to consider groups definable in o-minimal setting. Definable extreme amenability and $fsg$ are the abstract model-theoretic properties of the corresponding torsion-free and definably compact groups that allowed us to solve these cases. Generally, when trying to describe definable topological dynamics of a group $G$ decomposed into $G = KH$ using the induced action of $H$ on $K$ (in a way described in Section 4), we are forced to work with both left and right translations of types (and their coheirs) in $K$, but only left translations of types (and their heirs) in $H$. Assumption of $fsg$ for $K$ guarantees the existence of two-sided minimal ideals (over any model) that are all finitely satisfiable extensions of their restrictions. Assumption of definable extreme amenability of $H$ can likely be weakened.

The $fsg$ property can be viewed as a model-theoretic abstraction of compactness for definable groups, and is related to the notion of compact domination. Consider that compact-torsion-free decomposition in the o-minimal setting is the model-theoretic analog of Iwasawa decomposition for semisimple Lie groups. The classic theorem states that a semisimple real Lie group $G$ decomposes as $G = KAN$ with $K$ compact, $A$ abelian and $N$ nilpotent such that $H = AN$ is also a group. Variants of Iwasawa decomposition exist in other contexts. For instance, for a local non-Archimedean field $F$ with discrete valuation, the group $GL_n(F)$ can be decomposed into $GL_n(\mathcal{O}_F)U_n(F)$, where $\mathcal{O}_F$ is the ring of integers of $F$ and $U_n$ the group of upper-triangular matrices. In this case, the group $GL_n(\mathcal{O}_F)$ is a maximal compact subgroup of $GL_n(F)$. The field $\mathbb{Q}_p$ of p-adic numbers has been studied model-theoretically. Particularly, the theory of $\mathbb{Q}_p$ in the so-called Macintyre language has been shown to have many nice properties ($NIP$, elimination of quantifiers, definable choice). A notion of definable compactness for definable sets exists, and definably compact groups are shown to have $fsg$.

Consider the group $GL_n(\mathbb{Q}_p)$ with the decomposition $GL_n(\mathbb{Z}_p)U_n(\mathbb{Q}_p)$. The group $GL_n(\mathbb{Z}_p)$ is $fsg$, and while in this case the intersection of $GL_n(\mathbb{Z}_p)$ and $U_n(\mathbb{Q}_p)$ is not trivial, it is also $fsg$, giving hope it can be easily dealt with. The group $U_n(\mathbb{Q}_p)$ is not definably extremely amenable. It is however definably amenable. Attempting to generalize our result to the case where $G = KH$ for $K$ $fsg$ and $H$ definably amenable (with possibly nontrivial intersection) to tackle the $p$-adic setting may require additional assumptions. In the general case of definably amenable groups, the role of almost periodic types is assumed by so-called $f$-generic types. Since, the dynamical analysis of $\Sext{G}{M}$ involves translations of heirs of types in $H$, a reasonable assumption on $H$ might be that it contains an $f$-generic external type such that for any $N \succ M$ with $|N| = |M|$, $p|N$ is also $f$-generic. This is the case for the group $U_n(\mathbb{Q}_p)$.

\small{

}


\begin{thebibliography}{9}

\bibitem{CPS}
	A. Chernikov, A. Pillay, P. Simon,
	\textit{External definability and groups in NIP theories},
	Journal of the London Mathematical Society, 90(1):213–240, 2014

\bibitem{CS}
	A. Chernikov, P. Simon,
	\textit{Definably amenable NIP groups},
	preprint

\bibitem{Con1}
	A. Conversano,
	\textit{Maximal compact subgroups in the $o$-minimal setting},
	Journal of Mathematical Logic, 13 (2013), pp. 1-15

\bibitem{Con2}
	A. Conversano,
	\textit{A reduction to the compact case for groups definable in $o$-minimal structures},
	J. Symbolic Logic, 79 (2014), pp. 45-53.

\bibitem{CP}
	A. Conversano, A. Pillay,
	\textit{Connected components of definable groups and o-minimality I},
	Advances in Mathematics, 231 (2012), pp. 605-623

\bibitem{GPP}
	J. Gismatullin, D. Penazzi, A. Pillay,
	\textit{Some model theory of $SL(2,\mathbb{R})$},
	Fundamenta Mathematicae 229(2) (2012)

\bibitem{HP}
	E. Hrushovski, A. Pillay,
	\textit{On NIP and invariant measures}, 
	J. Eur. Math. Soc. (JEMS), 13(4):1005–1061, 2011

\bibitem{HPP}
	E. Hrushovski, Y. Peterzil, A. Pillay,
	\textit{Groups, measures, and the NIP},
	J. Amer. Math. Soc., 21(2):563–596, 2008

\bibitem{Jag}
	G. Jagiella,
	\textit{Definable topological dynamics and real Lie groups},
	Math. Logic Quarterly, 61 (2015), 4555

\bibitem{New2}
	L. Newelski,
	\textit{Model theoretic aspects of the Ellis semigroup},
	Israel Journal of Mathematics, August 2012, Volume 190, Issue 1, pp. 477-507

\bibitem{New1}
	L. Newelski,
	\textit{Topological dynamics of definable group actions},
	J. Symbolic Logic, 74, Issue 1 (2009), pp. 50-72

\bibitem{New3}
	L. Newelski,
	\textit{Topological dynamics of stable groups},
	J. Symbolic Logic, 79, Issue 4 (2014), pp. 1199-1223

\bibitem{PS}
	Y. Peterzil, C. Steinhorn,
	\textit{Definable compactness and definable subgroups of ominimal groups},
	J. London Math. Soc. 59 (1999), 769–786

\bibitem{PY}
	A. Pillay, N. Yao,
	\textit{On minimal flows, definably amenable groups, and o-minimality},
	Advances in Mathematics, 290:483–502, 2016

\bibitem{She}
	S. Shelah,
	\textit{Dependent first order theories, continued},
	Israel J. Math., 173 (2009), pp. 1-60

\bibitem{YL}
	N. Yao and D. Long,
	\textit{Topological dynamics for groups definable in real closed field},
	Annals of Pure and Applied Logic, 166(2015), 261-273.

\end{thebibliography}
\end{document}